\documentclass[12pt]{amsart}

\newtheorem{theorem}{Theorem}[section]

\newtheorem{corollary}[theorem]{Corollary}

\newtheorem{lemma}[theorem]{Lemma}

\theoremstyle{definition}

\newtheorem{definition}[theorem]{Definition}

\newtheorem{remark}[theorem]{Remark}

\newtheorem{example}[theorem]{Example}

\theoremstyle{parrafo}

\usepackage{graphicx}
\usepackage{amssymb}
\usepackage{amsthm}
\usepackage{mathtools}
\usepackage{bbm}
\usepackage{xcolor}

\newcommand{\N}{\mathbb{N}}

\newcommand{\Q}{\mathbb{Q}}
\newcommand{\R}{\mathbb{R}}

\begin{document}

\author{J. M. Aldaz and A. Caldera}
\address{Instituto de Ciencias Matem\'aticas (CSIC-UAM-UC3M-UCM) and Departamento de 
Mate\-m\'aticas,
Universidad  Aut\'onoma de Madrid, Cantoblanco 28049, Madrid, Spain.}
\email{jesus.munarriz@uam.es}
\email{jesus.munarriz@icmat.es}
\address{Departamento de Mate\-m\'aticas, Universidad de Oviedo, 33007, Oviedo, Spain}
\email{alberto.caldera.morante@gmail.com}


\thanks{The first named author was partially supported by Grant PID2019-106870GB-I00 of the
MICINN of Spain, by  V PRICIT (Comunidad de Madrid - Spain), and also  by ICMAT Severo Ochoa project 
CEX2019-000904-S (MICINN)}

\thanks{The second named author was supported by ICMAT Severo Ochoa project 
	CEX2019-000904-S (MICINN)}

	\title[]{Boundedness properties of modified averaging operators and geometrically doubling metric spaces}

	
	\maketitle

\begin{abstract} We characterize the geometrically doubling condition of a metric space in
terms of the  uniform $L^1$-boundedness of superaveraging operators, where uniform refers to the
existence of bounds independent of the measure being considered.
\end{abstract}

	
\section {Introduction} 
	It is shown in  \cite{Al1} that averaging operators are  $L^1$-bounded on geometrically doubling metric measure spaces. We complement this result,  by presenting 
 a sharp variant of it.
 Call an operator superaveraging (cf. Definition \ref{modaver})
 if integration is performed over balls larger than the ones appearing in the denominators. On geometrically doubling spaces, superaveraging operators are always $L^1$-bounded, with bounds that depend on the geometrically doubling constant of the space and on the expansion factor $t$, but can be taken to be independent of the measure $\mu$, cf. Theorem \ref{geomdoubling}. Furthermore, for any expansion factor $t  > 1$, this type of boundedness entails that the space is geometrically doubling, so 
 in fact we have an operator-theoretic characterization of the geometrically doubling condition.
				
 Boundedness on $L^1 (\mu)$ of a superaveraging operator does not in general  imply boundedness on $L^p (\mu)$ for $p > 1$, cf. Example \ref{smallballs}, so to apply interpolation arguments, it is interesting to know when boundedness on $L^\infty (\mu)$ holds. We shall see that $L^\infty(\mu)$-boundedness of the superaveraging operators for every $r > 0$ and every $t >1$, is equivalent to $\mu$ being doubling almost everywhere, cf. Corollary \ref{doub}.
But in that case the $L^p$ boundedness for $1\le p \le \infty$ follows directly from doubling a.e.; thus,
 $L^1$-$L^\infty$ interpolation is not useful in this context. 

	\section {Definitions and notation} 
	
	We will use $B^{\operatorname{o}}(x,r) := \{y\in X: d(x,y) < r\}$ to denote metrically open balls, \textit{}
	and 
	$B^{\operatorname{cl}}(x,r) := \{y\in X: d(x,y) \le r\}$ to refer to metrically closed balls (``open balls" and ``closed balls" will always be understood in the metric, not the
	topological sense). 
	If we do not want to specify whether balls are open or closed,
	we write $B(x,r)$. But when we utilize $B(x,r)$, we assume that all balls are of the same kind, i.e., all open or all closed.  
	To avoid trivialities, we will always suppose that measures are not identically zero and that metric spaces contain at least two points.

	\begin{definition} A Borel measure is   {\em $\tau$-additive} or {\em $\tau$-smooth}, if for every
		collection  $\{U_\alpha : \alpha \in \Lambda\}$
		of  open sets, 
		$$
		\mu (\cup_\alpha U_\alpha) = \sup_{\mathcal{F}} \mu(\cup_{i=1}^nU_{\alpha_i}),
		$$
		where the supremum is taken over all finite subcollections $\mathcal{F} = \{U_{\alpha_1}, \dots, U_{\alpha_n} \}$ of  $\{U_\alpha : \alpha \in \Lambda\}$.
		We say that $(X, d, \mu)$ is a {\em metric measure space} if
		$\mu$ is a  $\tau$-additive  Borel measure on the metric space $(X, d)$, such that $\mu$ assigns finite measure to bounded Borel sets.
	\end{definition} 
We are only interested in locally bounded Borel measures, since we want balls to have finite measure, so from now on we assume that this is always the case. 
Next  we present some motivation regarding our use of $\tau$-additive measures.
Trivially all Borel measures on a separable metric space are $\tau$-additive, and likewise all Radon measures (that is, measures that are inner regular with respect to the compact sets) on arbitrary metric spaces. In fact, if the standard axioms ZFC for set theory  are consistent, then it is consistent to assume that all
locally finite measures are $\tau$-additive (see \cite[Proposition 7. 2. 10]{Bo} for more details). 

All of this means that  the assumption of $\tau$-additivity does not rule out from the definition of metric measure spaces any metric space that can be proven to exist using ZFC. Note, for instance, that the requirement of separability  excludes spaces such as $X = L^\infty ([0,1], \lambda)$, where $\lambda$ is the one-dimensional Lebesgue measure, and $d(f,g) = \|f - g\|_\infty$. For us this is a perfectly good metric space. Note also that in the theory of stochastic processes, probabilities on spaces of functions come up naturally. So there is no a priori reason to remove from consideration  ``large'' metric spaces of functions endowed with locally finite measures, or even just probabilities. 
 
Next we recall the standard example of a measure that is not $\tau$-additive, under the assumption that measurable cardinals do exist (measurable cardinals  are so large that their existence cannot be proven from the standard axioms ZFC of set theory, unless these axioms are  inconsistent; more precisely, measurable cardinals serve as models of ZFC, therefore proving ZFC's consistency, so G\"odel's Second Incompleteness Theorem applies).
Let $X$ have measurable cardinality, and let $d$ be  the standard $\{0, 1\}$-valued discrete metric, so all subsets are open. By definition of measurable cardinal there exists a $\{0, 1\}$-valued Borel measure  that takes the value 0 on each singleton, and assigns measure 1 to $X$. Clearly all finite unions of singletons have measure zero and the union of all singletons has measure 1.

\vskip .2 cm

	Recall that 
	the complement of
	the support $(\operatorname{supp}\mu)^{\operatorname{c}} := \cup \{ B^{\operatorname{o}}(x, r): x \in X, \mu (B^{\operatorname{o}}(x,r)) = 0\}$
	of a Borel  measure
	$\mu$,  is an open set, and hence measurable. 
	
	\begin{definition}\label{supp} Let $(X, d)$ be a metric space and let
		$\mu$ be a locally finite Borel measure on $X$. 
		If $\mu (X \setminus \operatorname{supp}\mu) = 0$, 
		we say that $\mu$ has {\em full support}. 
	\end{definition}
	
	By $\tau$-additivity,  if $(X, d, \mu)$ is  a metric measure space, then 
	$\mu$ has full support,
	since $X \setminus \operatorname{supp}\mu $ is a union of open balls of measure zero.
	Actually, the other implication also holds, for the support is always separable,  so  having full support is equivalent to
	$\tau$-additivity (cf. \cite[Proposition 7. 2. 10]{Bo} for more details).

		\begin{definition}\label{modaver} Let  $(X, d, \mu)$ be a metric measure space and let $g$ be  a locally integrable function 
	on $X$. The modified 
	averaging operators $A_{t, r, \mu}$ acting on $g$ are defined as follows: for each pair of positive numbers $t, r > 0$ and each $x\in \operatorname{supp}\mu$, set
	\begin{equation}\label{avop}
		A_{t, r , \mu} g(x) := \frac{1}{\mu
			(B(x, r))} \int _{B(x, tr)}  g (y) \ d\mu (y).
	\end{equation}
We shall speak of superaveraging operators if $t > 1$,  and of  subaveraging  operators when $t < 1$.
\end{definition}

\begin{remark} 
Note that  modified averaging operators $A_{t, r, \mu}$ acting on, say $g$, yield functions defined  almost everywhere,  by
$\tau$-additivity: if for some $y\in X$ the function $A_{t, r, \mu} g$ is not defined at $y$, it is because $\mu (B(y,r) ) = 0$, from whence it follows that $\mu (B^{\operatorname{o}}(y,r)) = 0$, and thus $y \in X \setminus \operatorname{supp}\mu $. 
\end{remark}
	
	 Sometimes it is convenient to  specify whether balls are open or closed; in that case,
	we use $A_{t,r , \mu}^{\operatorname{o}} $ and $A_{t,r , \mu}^{\operatorname{cl}} $ for the corresponding operators. 
	Furthermore, when we are considering only one measure $\mu$ we often omit it, 
	writing $A_{t, r } $ instead of the longer $A_{t, r , \mu} $. 
	
	\section{Some background results}
	
		\begin{definition} \label{conjugate} Let $t > 0$. We call 
		\begin{equation*}\label{conjug}
			a_{t,r, \mu} (y)  
			: =
			\int_{X}  \frac{\mathbf{1}_{B(y,tr)}(x)}{\mu (B(x,r))}  \  d\mu(x) 
		\end{equation*}
		the {\em conjugate function} to the modified averaging operator $A_{t, r, \mu}$.
	\end{definition}
As before, when we consider only one measure $\mu$,  we will often 
write $a_{t, r } $ instead of $a_{t, r , \mu} $.

 \vskip .2 cm

Averaging operators are bounded on $L^1(\mu)$ if and only if the corresponding conjugate functions $a_{r, \mu} \in L^{\infty}(\mu)$, by \cite[Theorem 3.3]{Al1}.  The same argument (which   we include for the reader's convenience) shows that this characterization also holds for the modified averaging operators.

\begin{theorem}\label{Thrmconj} 
	Let $t > 0$ and let $(X,d,\mu)$ be a metric measure space. The averaging operator $A_{t,r}$ is bounded on $L^1(\mu)$ if and only if $a_{t,r} \in L_{\infty}(\mu)$, in which case $\left \|  A_{t,r}\right \|_{L^1(\mu) \rightarrow L^1(\mu)} = \left \| a_{t,r} \right \|_{\infty}$.
\end{theorem}

\begin{proof}
	Let $0 \le f \in L^1(\mu)$, and suppose  $a_{t,r} \in L^\infty (\mu)$. Since by
	Fubini-Tonelli
	\begin{equation*}\label{fubini}
		\|A_{t,r} f\|_{L^1} 
		=
		\int_X A_{t,r} f(x)  \  d\mu(x)
		=
		\int_X\int_X  \frac{\mathbf{1}_{B(x,rt)}(y) }{\mu (B(x,r))} f (y) \  d\mu(y) \  d\mu(x)
	\end{equation*}
	\begin{equation*}\label{fubini2}
		= 
		\int_X  f (y) \int_X  \frac{\mathbf{1}_{B(y,rt)}(x)}{\mu (B(x,r))}  \  d\mu(x) \  d\mu(y)
		= 
		\int_X  f (y) \ a_{t,r}(y) \  d\mu(y),
	\end{equation*}\\
	it follows from H\"older's inequality that $\|A_{t,r} \|_{L^1(\mu)\to L^1(\mu)} \le \|a_{t,r}\|_\infty. 
	$

	On the other hand, we claim that if $\|A_{t,r} \|_{L^1(\mu)\to L^1(\mu)} \le C$, then $\|a_{t,r}\|_\infty
	\le C. 
	$
	Towards a contradiction, suppose $C < \|a_{t,r}\|_\infty$ (including the case $\|a_{t,r}\|_\infty
	= \infty$). Then there is a $\lambda > C$ and a measurable set $A_\lambda$ such that 
	$A_\lambda \subset \{ a_{t,r}> \lambda\}$ and $0 < \mu (A_\lambda)  < \infty$. Let $f:= \mathbf{1}_{A_\lambda} \in L^1(\mu)$.
	Then 
	\begin{equation*}\label{fubini3}
		\|A_{t,r} f\|_{L^1} 
		=
		\int_X  f (y) \  a_{t,r}(y) \  d\mu(y)
		>  \int_{A_\lambda}  \lambda  \  d\mu(y)
		= \lambda \mu (A_\lambda) > C \|f\|_{L^1},
	\end{equation*}
thereby contradicting the assumption $\|A_{t,r} \|_{L^1(\mu)\to L^1(\mu)} \le C$.
\end{proof}

\begin{corollary}\label{sub} 
	Let $t \in (0,1/2]$ and let $(X,d)$ be a metric space. Then for every locally finite, $\tau$-additive Borel measure $\mu$ on $X$, the subaveraging operator $A_{t,r, \mu}$ satisfies  $\left \|  A_{t,r, \mu}\right \|_{L^1(\mu) \rightarrow L^1(\mu)} \le 1$.
\end{corollary}

\begin{proof} Note that if $\mathbf{1}_{B(y,rt)}(x) = 1$, then $B(y,rt) \subset B(x,r)$, since $t \le 1/2$. Thus, for every
$y \in X$,  we have $a_{t,r, \mu} (y) \le 1$.
\end{proof}

	\section{Geometrically doubling metric spaces}
	
	\begin{definition} Let $(X, d)$ be a metric space. A {\em strict $r$-net} (resp. {\em non-strict $r$-net}) 
		in $X$ is a subset $S \subset X$ such that for any pair of distinct points $x,y \in S$, we have $d(x,y) > r$  (resp. $d(x,y)  \ge r$). 
	\end{definition} 
	
	We  speak of an {\em $r$-net}   if we do not want to specify whether it is strict or not.
	To ensure disjointness of the balls $B(x,r/2)$,  $r$-nets are always  taken to be strict when working with closed balls;
	otherwise, we assume $r$-nets are
	non-strict.
	
	\begin{definition} \label{geomdoub} A metric space is {\it geometrically doubling}  if there exists a positive
		integer $D$ such that every ball of radius $r$ can be covered with no more than $D$ balls
		of radius $r/2$.  We call the smallest such $D$ the {\em geometrically doubling constant} of the space.
	\end{definition}
	
		We use $D^{\operatorname{o}}$ and $D^{\operatorname{cl}}$  to refer to the geometrically doubling constants for open and for closed balls.
	It is easy to see, by enlarging balls slightly, that the geometrically doubling condition is satisfied for open balls if and only if
	it is satisfied for closed balls, but the constants will in general be different. 
	By analogy with previous notation, we use  $M^{\operatorname{cl}}$   and $M^{\operatorname{o}}$ for 
	the maximum sizes of strict and non-strict $r$-nets
	respectively, in balls of radius $r$.

	\begin{remark} \label{rnetbounds}
		Let $X$ be geometrically doubling with constant $D$, and let $M_{2}$ be the maximum cardinality of an $r/2$-net in $B(x, r)$,  taken over 
		all $x\in X$ and all  $r > 0$. It easily follows from a maximality argument that $D \le M_{2}$: the balls of radius $r/2$ centered at each of the points of a maximal $r/2$-net must cover all of $B(x,r)$, for any point not covered could be added to the supposedly maximal $r/2$-net. On the other hand, if $M$ is the maximum cardinality of an $r$-net in $B(x, r)$,  taken over 
		all $x\in X$ and all  $r > 0$, then $M \le D$: a ball of radius $r/2$ can contain at most one point of an $r$-net. 
\end{remark}
				
\begin{example} With the preceding notation, it may happen that $D < M_{2}$. Consider the open balls case on $\R$. Note that 
		$D^{\operatorname{o}} = 3$, since $(-1,1) = \{0\} \cup (-1,0) \cup (0,1)$, so $(-1,1)$ can be covered with 3 balls of radius $1/2$, and no less than 3. However, $M^{\operatorname{o}}_{2} >3$: the set $\{-6/10, -1/10, 4/10, 9/10\}$ is a $1/2$-net in $(-1,1)$, and the
		balls of radius 1/2 centered at these points 
		give us a cover of $(-1,1)$ for which the minimal cardinality is not obtained.
\end{example}

	\begin{example} \label{geomdoub} To clarify 
		the meaning of the geometrically  doubling constant $D$,  note 
		first that if $D = 1$ then $X$ has only one point, 
		a case excluded by assumption,
		so $D \ge 2$.
		
		Consider the special case $X= \mathbb{R}^d$ with the euclidean distance, 
		and $\mu = \lambda^d$, the $d$-dimensional Lebesgue measure. By translation and dilation
		invariance, 
		it is enough to consider collections of translates  of  $B(0,1)$, such that the said collections  
		cover $B(0,2)$. When $d = 1$, it has already been noted in the previous example that $D^{\operatorname{o}} = 3$. On the other hand, for closed balls clearly $[-2,2] = [-2, 0] \cup [0,2]$, so $D^{\operatorname{cl}} = 2$.

When $d \ge 2$ we have that $2^d < D < 5^d$. 
The inequalities $2^d < D <  5^d$ follow from
		well known volumetric arguments.  Since $ \lambda^d (B(0,2)) = 
		2^d \lambda^d  B(0,1)$, at least $2^d$ translates of $B(0,1)$ are needed to cover $B(0,2)$.
		Furthermore, since there must be some overlap among the covering balls, we get 
		$2^d < D$. 
		
		To see why $D <  5^d$, let $\{v_1, \dots, v_N\}$ be a maximal  1-net in  $B(0,2) \subset \mathbb{R}^d$. The maximality of 
		$\{v_1, \dots, v_N\}$  entails that
		$B(0,2)   \subset   \cup_1^N B(v_i, 1)$, so $D \le N$. Since the balls $B(v_i, 1/2)$  
		are disjoint, contained in  $B(0, 5/2)$, and do not form a packing of $B(0, 5/2)$, we conclude that $N <  (5/2)^d/(1/2)^d = 5^d$.
		
 If we consider the $\ell^\infty$ norm on  $\mathbb{R}^d$ instead of the euclidean norm when $d \ge 2$,
		so balls are cubes with sides parallel to the axes,
		it is clear that $D^{\operatorname{cl}} = 2^d$, and not difficult to
		see (for instance, by an induction argument) that
		$D^{\operatorname{o}} = 3^d$.
	\end{example}

	Statements 1) and 2) of the next lemma can be found as part of \cite[Lemma 2.3]{Hy}. For ease in the applications we have added statements 3) and 4).
	The ceiling function 
	$\lceil  t \rceil$ is defined as the least integer $n$ satisfying $t\le n$.
	
	\begin{lemma}\label{hyt}  Let $(X, d)$ be a metric space and let $N$ be a fixed positive integer. 
		Each of the following statements implies the next:
		
		1) Every ball $B(x, r)\subset X$ can be covered with at most $N$ balls of radius $r/2$.

		2)  For all $t\in (0,1)$, every ball $B(x, r)\subset X$ can be covered with at most $N^{\lceil - \log_2 t \rceil}$ balls of radius $t r$.
		
		3) There exists a $t\in (1/2, 1)$ such that every ball $B(x, r)\subset X$ can be covered with at most $N$ balls of radius $t r$.
		
		4)  With $t\in (1/2,1)$ as in the preceding statement, every ball $B(x, r)\subset X$ can be covered with at most $N^{\lceil - 1/ \log_2 t \rceil}$ balls of radius
		$r/2$.
					\end{lemma}
	
	\begin{proof} Let us check that  1) implies 2). For
		$0 < t < 1$, let $k$ be the unique positive integer
		that satisfies 
		$2^{-k} \le t < 2^{-k + 1}$. By an inductive argument, every ball $B(x,r)$ can
		be covered with at most $N^{k}$ balls of radius $2^{-k} r \le t r$.
		Since $k \ge - \log_2 t > k - 1$, we have that 
		$k = \lceil - \log_2 t \rceil$. 
		
		Part 3) is a special case of part 2), since for $t\in (1/2, 1)$, we have  
		$1 = \lceil - \log_2 t \rceil$. We mention that part 3) is also immediate from part 1), just take a cover by balls of radius $r/2$ and enlarge the radii to $tr$. 
		
		Finally, for 3) implies 4) we argue as in 1) implies 2): let
	 $k$ be the unique positive integer
		that satisfies 
		$t^{k} \le 1/2  < t^{k - 1}$,  or equivalently, let $k := \lceil - 1/\log_2( t )\rceil$. After covering $B(x,r)$ with at most $N$ balls 	of the form
$B(w, t r)$, by hypothesis each one of these can also be covered by at most $N$ balls of the form $B(v , t^2 r)$. Iterating, we find that $B(x,r)$ can be covered by at most $N^{\lceil -1/\log_2(t )\rceil}$ balls of radius 
$r t^k \le r/2$. 
\end{proof}

Of course, part 4) is an immediate consequence of part 1); deriving part 4) directly from the information given in part 3) yields the same qualitative information (the space is geometrically doubling) 
but quantitatively weaker bounds.

\vskip .2 cm

The subaveraging operators are pointwise dominated by the averaging operators, and these
are uniformly (on $r >0$ and on $\mu$) bounded  when the metric space is geometrically doubling by Theorem \cite[3.5]{Al1}, so the next result deals only
with values of $t > 1$. Essentially it says that the uniform boundedness of the superaveraging operators for some $t > 1$, or for all $t > 1$, is equivalent to the geometrically doubling condition. This is not the case for the averaging operators, where the equivalence is given by a weaker property of Besicovitch type, cf. \cite[Theorem 4.7]{Al2}.

\begin{theorem}  \label{geomdoubling} Let $(X, d)$ be a  
 	metric  space, let $t > 1$ and let
$M_t$ be the supremum of all cardinalities of
  $r$-nets in $B(x, t r)$, where the supremum is taken
over all  $x \in X$ and all $r >0$. 

Each of the following statements implies the next one:

1) The metric  space  $(X, d)$ is geometrically doubling, with geometrically  doubling constant $D$.

2)  For every locally bounded $\tau$-additive Borel measure $\mu$, every $t > 1$ and every $s > 0$, the bound
$\|A_{t, s, \mu} \|_{L^1(\mu)\to L^1(\mu)} \le M_t \le D^{\lceil \log_2 2t \rceil}$ holds.

3)  There exist a $t \in (1,2) $ and a $C = C(t) > 0$ such that for every $s > 0$ and every finite  weighted sum of Dirac deltas $\mu := \sum_{i = 1}^N c_i \delta_{x_i}$, the bound
$\|A_{t, s, \mu} \|_{L^1(\mu)\to L^1(\mu)} \le C$ holds.
	
4) With $t$ and $C$ as in the previous statement, the metric  space  $(X, d)$ is geometrically doubling, with geometrically doubling constant bounded by $C^{\lceil 1/\log_2(t)\rceil}$.
 \end{theorem} 

 \begin{proof} For 1) implies 2)  we present only a sketch of the proof,  referring the interested reader to the almost identical proof of 
 Theorem \cite[3.5]{Al1}, where a more complete argument can be found. The only difference lies in the fact that since $t > 1$, instead of the bound
 $M \le D$ one obtains $M_t \le D^{\lceil \log_2 2 t \rceil}$. Some details follow. 

Without loss of generality suppose that $X = (\operatorname{supp}\mu)$.  Fix  $y \in X$, and let $0 < \varepsilon \ll  1$. We show that $a_{t,s}(y) \leq M_t$.
Noticing that  $b_1 \coloneqq \text{inf}\{  \mu (B(x,s))  : x \in B(y, ts)\} > 0$, we choose $u_1 \in B(y,ts)$ such that $\mu (B(u_1,s )) < (1 + \varepsilon)b_1$. Let $b_2 \coloneqq \text{inf}\{  \mu (B(x,s))  : x \in B(y,ts)\setminus B(u_1,s) \} $, and select $u_2 \in B(y,ts)\setminus B(u_1,s)$ so that $\mu (B(u_2,s)) < (1 + \varepsilon)b_2$; repeat until the process cannot be continued. This happens in   at most $M_t$ steps, since the selected centers $u_k$ form an $s$-net $S$ in $B(y, ts)$. Denote by $m$ the cardinality of $S$.
	For every $x \in B(y,ts)$ we have
	$$
	\frac{\mathbf{1}_{B(y,ts)}(x)}{\mu (B(x,s))} 
	\leq 
	(1+\varepsilon) \sum_{i=1}^{m} \frac{\mathbf{1}_{B(y,ts)\cap B(u_i,s)}(x)}{\mu (B(u_i,s))}.
	$$
	Integrating,
	$$
	a_{t,s}(y) = \int_{X} 	\frac{\mathbf{1}_{B(y,ts)}(x)}{\mu (B(x,s))} \hspace{2pt}d\mu (x)
	 \leq 
	\int_{X} (1+\varepsilon) \sum_{i=1}^{m} \frac{\mathbf{1}_{ B(u_i,s )}(x)}{\mu (B(u_i,s))} \hspace{2pt}d\mu (x) \leq (1+\varepsilon)M_t.
	$$
Letting $\varepsilon \downarrow 0$, we conclude from Theorem \ref{Thrmconj}  that $\left \|  A_{t,s}\right \|_{L^1(\mu) \rightarrow L^1(\mu)} \leq M_t$.
As in Remark 4.3, we recall that each ball of radius $s/2$ can contain at most one point of an $s$-net, so to bound 
$M_t$ from above it is enough to estimate the number of balls of radius $s/2$ needed to cover $B(y, ts)$. But by part 2) of Lemma \ref{hyt}, we know that $D^{\lceil \log_2 2 t \rceil}$ balls of radius $s/2$ suffice.
 
 Since 2) implies 3) is
 trivial, with $C = M_t$, we only need to show that 3) implies 4), so suppose there is a $t \in (1,2) $
 such that for all discrete measures $\mu$ with finite support, and all $s > 0$, we have $\|A_{t, s, \mu} \|_{L^1(\mu)\to L^1(\mu)} \le C$. By part 3) of Lemma \ref{hyt}, to obtain the geometrically doubling condition it is enough to show that one can cover the large ball with a fixed number of smaller balls, where the radius is contracted not necessarily by 1/2, but by any number strictly less than 1. The assumption $t < 2$ is for mere convenience: of course, if we have boundedness for some $t \ge 2$, we also have it for $t < 2$. 
 
 Note that before we can take an $s$-net of maximal cardinality, we need to show that it is not possible to have an infinite $s$-net inside a ball.  To this end, select a ball $B(x,r)$ and a finite $r/t$-net $S := \{y_1,\dots,  y_m\} \subset  B(x,r)$.  For $0 < c \ll 1$ define 
$\mu_c :=  c  \delta_x + \sum_{i = i}^m \delta_{y_i}$, and set $f_c := c^{-1} \mathbf{1}_{\{x\}}$.  Then $\|f_c\|_{L^1(\mu_c)} = 1$. Next we use the bounds satisfied by the superaveraging operators,  for the given value of $t \in (1,2) $ and $s = r/t$, so by hypothesis  $\|A_{t, r/t, \mu_c}  f\|_{L^1(\mu_c)} \le C.$ Since $S$ is an $r/t$-net, any  ball of radius $r/t$ centered at a point of $S$ contains only one point of $S$, namely its center. These balls also might contain $x$, so for all $j = 1, \dots , m$, we have 
$1 \le \mu_c (B(y_j, r/t)) \le 1 + c$. On the other hand, when we apply the expansion factor $t$ to the radius $r/t$, we find that for every $y_j$, $j = 1, \dots , m$, we do have $x \in B(y_j, r)$. Thus
 	$$
 	1
 	\ge A_{t, r/t, \mu_c}  f (y_j)  \ge 
 	\frac{1}{ 1 + c},
 	$$
	so
	$$
	\frac{m}{ 1 + c} 
	\le
	\|A_{t, r/t, \mu_c} f_c\|_{L^1(\mu_c)}
	\le C,		
	$$
 	and we conclude that $m\le C$ by letting $c \to 0$. Now every finite $r/t$-net in $B(x,r)$ has cardinality bounded by $C$, so there are no  $r/t$-nets of infinite cardinality contained in $B(x,r)$. Choose a maximal $r/t$-net $S^\prime$
of points in $B(x,r)$, with respect to the partial order given by inclusion. By maximality of $S^\prime$ we conclude that 
$\{B(w, r/t) : w\in S^\prime\}$ is a cover of $B(x,r)$, of cardinality bounded by $C$. Using part 4) of Lemma \ref{hyt},
we see that the geometrically  doubling constant of the space is bounded by $C^{\lceil 1/\log_2(t)\rceil}$.
 \end{proof}  
  
  The next corollary restates the ``extrapolation'' part of the preceding result:  boundedness of the superaveraging operators for one single $t > 1$ implies the boundedness for all $t >1$, thought 
of course the bounds will be larger for larger values of $t$.
  	
	\begin{corollary}  \label{extrapolation} Let $(X, d)$ be a  
		metric  space. If there exist a  $t > 1$ and a $C = C(t) > 0$ such that for every $s > 0$ and every finite  weighted sum of Dirac deltas $\nu := \sum_{i = 1}^N c_i \delta_{x_i}$, the bound
		$\|A_{t, s, \nu} \|_{L^1(\nu)\to L^1(\nu)} \le C$ holds, then for	every $t > 1$ there exists a constant $C(t) > 0$ such that for every locally bounded $\tau$-additive Borel measure $\mu$ and every $s > 0$, the bound
		$\|A_{t, s, \mu} \|_{L^1(\mu)\to L^1(\mu)} \le C(t)$ holds.	
		\end{corollary} 	


	\begin{example}\label{smallballs} Let $ p > 1$. The boundedness on $L^1(\mu)$ of a superaveraging operator does not imply  the boundedness on $L^p(\mu)$, as the following example shows.
		
	 It is immediate from the characterization using $r$-nets that every subset of a geometrically doubling metric space is geometrically doubling. Let $(X, d)$ be the following subset (hence geometrically doubling) of the euclidean plane with the inherited distance: define $X := \{0, 1/2\}\times \N$. Let $\mu$ be the counting measure on $\{1/2\}\times \N$,  and on $\{0\}\times \N$, set $\mu \{(0,n)\}  = 1/n$. 
	 Writing $f_n := \mathbf{1}_{\{(1/2,n)\}}$,
 for all $p\in [1, \infty]$ we have $\|f_n\|_{p} = 1$.
	 For convenience we shall use closed balls, with
	 $r = 1/4$ and $t= 2$, so when $n>0$, we have
 $A_{2, 1/4 , \mu}^{\operatorname{cl}} f_n (0,n) \ge n/2$. Clearly $A_{2, 1/4 , \mu}^{\operatorname{cl}}$ is not bounded on $L^\infty(\mu)$, and for $1 < p < \infty$, 
 $$
 \|A_{2, 1/4 , \mu}^{\operatorname{cl}} f_n\|_{L^p(\mu)} 
 \ge
 \left(\frac{n^p}{2^p n}\right)^{1/p}, 
 $$
 so $A_{2, 1/4 , \mu}^{\operatorname{cl}}$ is not bounded on $L^p(\mu)$ either.
 
 On the other hand, if $r \ge 1/2$, then $\mu (B^{\operatorname{cl}} (x,r)) \ge 1$ for all $x \in X$, so it easily follows that for every $t > 1$ and every $p \ge 1$ we have boundedness of $A_{t, r , \mu}^{\operatorname{cl}}$ on $L^p(\mu)$.
\end{example}

In view of the preceding example, it is natural to ask under which conditions is $A_{t, r , \mu}$
bounded on $L^\infty (\mu)$, so that interpolation can be carried out. Next we show that boundedness for all $t$ and $r$ is equivalent to $\mu$ being doubling almost everywhere.

			\begin{theorem}\label{prop2} Let $(X, d,\mu)$ be a  metric measure space. 
	Fix $t > 0$ and $r > 0$. Then $\|A_{t, r, \mu}\|_{L^{\infty}(\mu)  \to L^\infty (\mu)} = \left\| \dfrac{\mu (B(\cdot,tr))}{\mu (B(\cdot,r))} \right\|_{L^\infty (\mu)}$. 
\end{theorem}

\begin{proof} To see that  
$$
\left\| \dfrac{\mu (B(\cdot,tr))}{\mu (B(\cdot,r))} \right\|_{L^\infty (\mu)}\leq \|A_{t, r,\mu}\|_{L^\infty (\mu)  \to L^\infty (\mu)},
$$ just take $f \equiv 1$ and note that for every $x \in \operatorname{supp} \mu$, we have 
$$
A_{t, r ,\mu}f(x) = \dfrac{\mu (B(x,tr))}{\mu (B(x,r))}.
$$ 
In order to prove 
$$
\left\| \dfrac{\mu (B(\cdot,tr))}{\mu (B(\cdot,r))} \right\|_{L^\infty (\mu)}
\geq
 \|A_{t, r,\mu}\|_{L^\infty (\mu)  \to L^\infty (\mu)},
$$
select $f \in L^{\infty}(\mu)$ and note that for every $x \in \operatorname{supp} \mu$  we have 		
$$
		| A_{t,r}f(x)| \leq \frac{1}{\mu(B(x, r))} \int _{B(x, tr)}  |f| \ d\mu \leq \frac{\mu(B(x, tr))}{\mu(B(x, r))} \| f \|_{L^\infty (\mu)} \leq 	\left\| \dfrac{\mu (B(\cdot,tr))}{\mu (B(\cdot,r))} \right\|_{L^\infty (\mu)} \| f \|_{L^\infty (\mu)}.
		$$
\end{proof}

\begin{definition}
	Let $(X,d,\mu)$ be a metric measure space. Then  $\mu$ is  {\em doubling} if there exists a constant $C > 1$ such that for all $x \in X$ and all $r > 0$, we have $\mu(B(x,2r)) \leq C\mu(B(x,r))$. If there is a measure zero set $A$ such that the restriction of $\mu$ to $X \setminus A$ is doubling, we say that $\mu$ is {\em doubling almost everywhere}.
\end{definition}	

Examples of measures $\mu$ that are doubling almost everywhere but not doubling are easy to find: just take one  Dirac delta $\delta_x$ at any point  $x\in X$, or let $\lambda$ be the linear Lebesgue measure on the $x$-axis of $\R^2$. It is clear from the definition that  $\mu$ is  doubling almost everywhere if and only if  
$$
\left\| \  \underset{ r > 0}{\sup} \left\{\dfrac{\mu (B(\cdot,2r))}{\mu (B(\cdot,r))} \right\} \ \right \|_{L^\infty (\mu)} < \infty,
$$ 
and  by a standard iteration argument, the $2r$ in the numerator can be replaced by any $t r$, provided $t > 1$. Hence the following result is immediate from Theorem \ref{prop2}.

			\begin{corollary} \label{doub} Let $(X, d,\mu)$ be a  metric measure space. 
	Then  the superaveraging operators $A_{t, r, \mu}$ are bounded on     
	$L^{\infty} (\mu)$ for all $r > 0$ and all $t > 1$, if and only if $\mu$ is doubling almost everywhere. 
\end{corollary}


\begin{thebibliography}{WWW}
		
	
		
		\bibitem[Al1]{Al1} Aldaz, J. M.
		{\em Boundedness of averaging operators on geometrically doubling metric spaces.}  Ann. Acad. Sci. Fenn. Math. 44 (1) (2019), 497--503. Available at the Mathematics ArXiv. 
		
\bibitem[Al2]{Al2}  Aldaz, J. M.  {\em Kissing numbers and the centered maximal operator.} J. Geom. Anal. 31 (2021), no. 10, 10194--10214. Available at the Mathematics ArXiv.

		
		\bibitem[Bo]{Bo}  Bogachev, V.  I. {\em Measure Theory.}
		(2007), 
		Springer-Verlag. 
		
	 
	
	\bibitem[Hy]{Hy}  Hyt\"onen, Tuomas {\em A framework for non-homogeneous analysis on metric spaces, and the RBMO space of Tolsa.} Publ. Mat. 54 (2010), no. 2, 485--504.
	
		
		
	\end{thebibliography}
\end{document}